\documentclass[11pt]{amsart}
\usepackage{amsfonts,amsthm,latexsym,amsmath,amssymb,amsmath}
\usepackage{graphicx}
\usepackage{geometry}
\usepackage[colorlinks = true]{hyperref}
\usepackage{xypic}

\newtheorem{theorem}{Theorem}[section]
\newtheorem{proposition}[theorem]{Proposition}
\newtheorem{lemma}[theorem]{Lemma}

\newtheorem{corollary}[theorem]{Corollary}

\newtheorem{conjecture}[theorem]{Conjecture}
\theoremstyle{definition}
\newtheorem{example}[theorem]{Example}
\newtheorem{remark}[theorem]{Remark}

\newcommand{\cF}{\mathcal{F}}

\newcommand{\bC}{\mathbb{C}}
\newcommand{\bX}{\mathbb{X}}
\newcommand{\ee}{\end{equation}}
\newcommand {\al}{\alpha}
\newcommand {\la}{\lambda}
\newcommand {\om}{\omega}
\newcommand{\be}{\beta}

\newcommand{\bP}{\mathbb{P}}

\newcommand{\bN}{\mathbb N}
\newcommand{\bZ}{\mathbb Z}

\newcommand \rk {\mathrm{rk}}
\newcommand \HF {\mathrm{HF}}
\newcommand \HS {\mathrm{HS}}
\newcommand{\VP}{\mathrm{VP}}

\begin{document}
          \numberwithin{equation}{section}

          \title[On generic and maximal $k$-ranks of binary forms]
          {On generic and  maximal $k$-ranks of binary forms}

\author[S.~Lundqvist]{Samuel Lundqvist}
\author[A.~Oneto]{Alessandro Oneto}
\author[B.~Reznick]{Bruce Reznick}
\author[B.~Shapiro]{Boris Shapiro}

\address[S.~Lundqvist, B.~Shapiro]{Department of Mathematics, Stockholm University, SE-106 91
Stockholm,  Sweden} \email{samuel@math.su.se, shapiro@math.su.se}

\address[A.~Oneto]{Inria Sophia Antipolis M\'editerran\'ee (team Aromath), 06902 Sophia Antipolis, France} \email{alessandro.oneto@inria.fr}

\address[B.~Reznick]{Department of Mathematics, University of Illinois, Urbana, IL 61801, USA} \email{reznick@illinois.edu}

\dedicatory{To James Joseph Sylvester, a mathematician, a polyglot, and a poet}
\date{}


\begin{abstract} In what follows, we pose two general conjectures {about decompositions of homogeneous polynomials as sums of powers}. The first one ({suggested by} G.~Ottaviani) deals with the generic $k$-rank of complex-valued forms of any degree divisible by $k$ in any number of variables. The second one (by the fourth author) deals with the maximal $k$-rank of binary forms. We settle the first conjecture in the cases of two variables and the second in the first non-trivial  case of the $3$-rd  powers of quadratic binary forms. 
\end{abstract}

\maketitle

\section{Introduction}  
A vast and currently very active area of mathematical research dealing with additive decompositions of polynomials started with the following classical result on binary forms proven in 1851 by J.~J.~Sylvester\footnote{{James Joseph (Sylvester) was born to a Jewish family in London in 1814. His remarkably original and successful mathematical career only partially helped him overcome the pervasive anti-Semitism of his era. For more on his life, see \cite{P}.}} \cite{Syl1,Syl2}.
\begin{theorem}[Sylvester's Theorem]\label{th:Sylv}
\rm{(i)} A general binary form $p\in\bC[x,y]$ of odd degree $k=2s-1$ with complex coefficients can be written as
\begin{equation}\label{eq:odd}
 p(x, y) =\sum_{j=1}^s(\al_{j}x+\be_{j}y)^k, \text{ for some }\al_j,\be_j\in \bC.
\end{equation}
\noindent
\rm{(ii)} A general binary form $p\in\bC[x,y]$ of even degree $k=2s$ with complex coefficients can be written as 
\begin{equation}\label{eq:even}
 p(x,y)=\la x^k +\sum_{j=1}^s(\al_{j}x+\be_{j}y)^k, \text{ for some }\lambda,\al_j,\be_j\in \bC.
\end{equation}
\end{theorem}
\begin{figure}[h]
\includegraphics[scale=0.7]{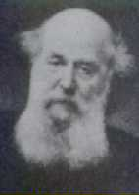}
\caption{ J.~J.~Sylvester around 1890.}
\label{fig1}
\end{figure}
After Sylvester's work, decompositions of polynomials into sums of powers of linear forms have been widely studied  from several perspectives starting with  the geometrical point of view by the classic  Italian school in algebraic geometry in the beginning of the 20-th century as well as current research by applied mathematicians and engineers  in connection with  tensor decompositions.

\medskip
Such presentations are often called {\it Waring decompositions} and, for a given polynomial $f$, the smallest length of such a decomposition is called the {\it Waring rank}, or simply, the {\it rank} of $f$. The minimal number of linear forms required to represent a general form of degree $k$ in $n$ variables as a sum  of their $k$-th powers is called the {\it generic rank} and denoted by $\rk^\circ(n,k)$, while the {\it maximal rank} $\rk^{\max}(n,k)$ is the minimal number of linear forms required to represent {\it any} form of degree $k$ in $n$ variables. Rephrasing Theorem \ref{th:Sylv} in this terminology, we have that 
$$
 \rk^{\circ}(2,k) = \left\lceil \frac{k+1}{2} \right\rceil.
$$
The explicit value of the generic {Waring} rank {for} any arbitrary $k$ and $n$ was obtained in the celebrated result of J.~Alexander and A.~Hirschowitz \cite{AH}. Except for the case of quadrics in all dimensions and four additional exceptions 
{$(n,k) = (3,4), (4,4), (5,3), (5,4)$}, the generic rank coincides with its expected value given by 
$$
 \rk^{\circ}(n,k) = \left\lceil \frac{1}{n}\binom{n+k-1}{k} \right\rceil.
$$
{In the case of quadrics, the generic rank is equal to $n$, while} in all the exceptional cases the generic rank is by $1$ bigger than the latter expected value.  

\medskip
{
	Additionally, the maximal Waring rank $\rk^{\max}(2,k)$ of binary forms equals $k$. This was probably a classical result, but it has been recently proved in \cite{CS}. Also, the maximal value $k$ is attained exactly on binary forms of the type $l_1l_2^{k-1},$ where $l_1$ and $l_2$ are linearly independent linear forms, see \cite[Theorem 5.4]{Re2}.
}

\medskip
Other types of additive decompositions of polynomials have been considered in the last decades. In \cite{FOS}, the fourth author jointly with R.~Fr\"oberg and G.~Ottaviani considered, for any triple of positive integers $(k,d,n)$ with $k,n \geq 2$, decompositions of homogeneous polynomials of degree $kd$ in $n$ variables as sums of $k$-th powers of forms of degree $d$. Given a form $f$ of degree $kd$, the smallest length of such a decomposition  denoted by $\rk_k(f)$ is called the {\it $k$-rank} of $f$. Analogously to the classical Waring rank, we define the {\it generic $k$-rank} for forms of degree $kd$ in $n$ variables, denoted by $\rk_k^{\circ}(n,kd)$, and the corresponding {\it maximal $k$-rank}, denoted by $\rk^{\max}_k(n,kd)$.

\medskip
The main result of \cite{FOS} is the following upper-bound on the generic $k$-rank, for any triple $(k,d,n)$, 
$$
 \rk^\circ_k(n,kd) \leq k^n.
$$
A remarkable property of this bound is its independence of the parameter $d$ and  also  its sharpness for any fixed $k$ and $n,$ when $d \gg 0$. The following  general conjecture about  $\rk^\circ_k(n,kd)$ was suggested  by G.~Ottaviani in 2013 (private communication). For $k=2$, it also coincides with \cite[Conjecture 1]{LSVB}.
\begin{conjecture}\label{conj:otta} For any triple $(k,d,n)$ of positive integers with $k,n,d \geq 2$, 
\begin{equation}\label{eq:main}
\rk^\circ_k(n,kd)=\begin{cases} \min \left\{s\ge 1 ~|~ s\binom{d+n-1}{n-1}-\binom {s}{2}\ge \binom{2d+n-1}{n-1}\right\}, & \text{ for } k=2;\\
\min \left\{s\ge 1 ~|~ s\binom{d+n-1}{n-1}\ge \binom{kd+n-1}{n-1}\right\}, & \text{ for } k\ge 3.
\end{cases}
\end{equation}
\end{conjecture}
Conjecture~\ref{conj:otta} is  supported by  substantial computer experiments performed by G.~Ottaviani and the present authors, see e.g. \cite{One}. A proof of Conjecture~\ref{conj:otta} would extend the Alexander--Hirschowitz Theorem (which corresponds to $d=1$) to the case of arbitrary triples $(k,d,n)$ and  will complete the important project of determining generic ranks for symmetric tensors. 

{
In this paper, we
}{mainly consider the case of binary forms.}

\setcounter{section}{2}
\setcounter{theorem}{2}
\begin{theorem}\label{th:genkrank}
For $k,d \geq 2$, the generic $k$-rank of binary forms of degree $kd$ is
 $$\rk^\circ_k(2,kd)=\left\lceil\frac{{kd+1}}{d+1}\right\rceil.$$
\end{theorem} 
This result extends Theorem~\ref{th:Sylv} to presentations of general binary forms of degree $kd$ as sums of $k$-th powers of binary forms of degree $d$  and gives a proof of Conjecture~\ref{conj:otta} in the case of binary forms. 

\setcounter{section}{1}

An alternative proof of Theorem~\ref{th:genkrank} using canonical forms can be found in \cite[Theorem 1.8]{Re1}. Our approach puts this result in a much more general setting{; in particular, we relate Conjecture \ref{conj:otta} to Fr\"oberg's Conjecture on Hilbert series of generic ideals. This {relation extends} to higher dimensions, as we explain in Appendix \ref{app:Froberg}, where we also settle the case of sums of squares in three variables $(k = 2, n = 3)$.}

\medskip
In Fall 2014, the fourth author formulated the following conjecture about the maximal rank of binary forms. 

\begin{conjecture}\label{conj:maxkrank}
For $k\geq 2$, the maximal $k$-rank of binary forms of degree $kd$ is
 $$\rk^{\max}_k(2,kd)=k.$$
\end{conjecture}

{As we said, the case $d = 1$ is classical and well-known, see \cite[Theorem 4.9]{Re2}. Moreover, by \cite[Theorem 5.4]{Re2}, we know that a binary form has maximal rank if and only if it can be decomposed as $\ell_1\ell_2^{k-1}$. Also, it is easy to prove that any binary polynomial of even degree can be decomposed as a sum two squares, see \cite[Theorem 5]{FOS}. In this paper, we prove Conjecture~\ref{conj:maxkrank} in the first open case of binary sextics decomposed as sums of cubes ($k=3, d = 2, n = 2$).}

\setcounter{section}{3}
\setcounter{theorem}{0}
\begin{theorem}\label{th:sextic}
Every binary sextic can be written as a sum of at most three cubes of binary quadratic forms.
\end{theorem}

\setcounter{section}{1}

{
 One can also suspect that $\rk^{\max}_k(l_1 l_2^{kd-1}) = k$, where $l_1$ and $l_2$ are non-proportional linear binary forms, {similarly to what happens} in the classical case. From \cite{CO}, we know that this is an upper bound, but computing the actual $k$-rank is a very difficult task. In Section \ref{sec: computing ranks}, we explain how our geometric approach can help in these computations; e.g., we can show that $\rk^{\max}_4(l_1 l_2^{7}) = 4$. In Appendix \ref{appendix: monomials}, we continue the work started in \cite{CO} and we prove {that a class of monomials has} $k$-rank smaller than $k$. 
}

\medskip
\noindent
\emph{Acknowledgements.} We want to thank  the participants of the problem-solving seminar in commutative algebra at Stockholm University and especially  Ralf Fr\"oberg for creating  a nice research atmosphere.  It is a pleasure to acknowledge the importance of our communication with Giorgio Ottaviani for the present study.  Finally, the fourth author is sincerely grateful to the University of Illinois, Urbana-Champaign for the hospitality in summer 2015 when a part of this project was carried out. 

\section{Generic $k$-ranks}\label{General rank}
\noindent In this section we focus on {\it generic $k$-ranks}.
   Let {$S = \bigoplus_{d\geq 0}S_d$} be the standard graded polynomial ring in $n$ variables with complex coefficients, where $S_d$ denotes the vector space of degree $d$ forms. By a count of parameters, we have a lower bound for the generic $k$-rank of forms of degree $kd$, i.e.,
$$\rk^\circ_k(n,kd)\ge \left\lceil\frac{\dim_{\bC} S_{d}}{\dim_{\bC}S_d} \right\rceil=\left\lceil \frac{{kd+n-1 \choose n-1}}{{d+n-1 \choose n-1}}\right\rceil.$$

\subsection{Secant varieties and Terracini's Lemma.}
A common approach to analyse Waring-type problems is to study {\it secant varieties}. We describe it in our context. 

Let $\VP_{n,kd}^{(k)}$ be the {\it variety of $k$-th powers} in the space of polynomials of degree $kd$, i.e.,
$$
 \VP_{n,kd}^{(k)} = \left\{[g^k] ~|~ g\in S_d\right\} \subset \bP(S_{kd}).
$$ 
In the case where $d = 1$, this is the classical {\it Veronese variety} ${\rm Ver}_{n,k}$. 

The {\it $s$-th secant variety} of $\VP_{n,kd}^{(k)}$ is the closure of the set of points that lie on the linear span of $s$ points on the variety of powers, i.e.,
$$
 \sigma_s \left(\VP_{n,kd}^{(k)}\right)  =  \overline{\bigcup_{g_1,\ldots,g_s \in S_d}\left\langle [g^k_1],\ldots, [g_s^k]\right\rangle} \subset \bP(S_{kd}),
$$
in other words, the $s$-th secant variety is the closure of the set of homogeneous polynomials of degree $kd$ with $k$-th rank {at most} $s$. Observe that, since $\VP_{n,kd}^{(k)}$ is non-degenerate, secant varieties give a filtration and eventually fill the ambient space. These definitions {allow us to} describe the $k$-th generic rank as
$$
 \rk^\circ_k(n,kd) = \min\left\{ s ~|~ \sigma_s \left(\VP_{n,kd}^{(k)}\right)= \bP(S_{kd})\right\}.
$$
Therefore, one way to understand generic $k$-ranks is to compute the dimensions of secant varieties. In order to do so, a classical tool is Terracini's Lemma which describes the generic tangent space to a {\it secant variety}, see \cite{Te}. In our setting, it is equivalent to the following statement.
\begin{lemma}[Terracini's Lemma]\label{Terracini}
 Let $P$ be a generic point on the linear span of $s$ generic points on the variety of $k$-th powers, $[g_1^k],\ldots,[g_s^k]\in \VP_{n,kd}^{(k)}$. Then the tangent space to $\sigma_s\left(\VP_{n,kd}^{(k)}\right)$ at $P$ coincides with the linear space of the tangent spaces to $\VP_{n,kd}^{(k)}$ at the $s$ points, i.e., 
 $$
  T_P \sigma_s\left(\VP_{n,kd}^{(k)}\right) = \left\langle T_{[g_1^k]}\VP_{n,kd}^{(k)},\ldots,T_{[g_s^k]}\VP_{n,kd}^{(k)}\right\rangle \subset \bP(S_{kd}).
 $$
\end{lemma}
By Lemma~\ref{Terracini}, we can reduce the problem of computing generic ranks to a question in commutative algebra. Let us recall some basic definitions.

Given a homogeneous ideal $I\subset S$, we have an induced grading of $I$ and of the quotient ring $S/I$. We define the {\it Hilbert function} of $S/I$ {{\it in degree $d$} as {the dimension of the vector space formed by} the degree $d$ part of $S/I$, i.e.,}
$$\HF_{S/I}(d)=\dim_{\bC}[S/I]_d=\dim_{\bC}S_d-\dim_{\bC}I_d,~~\text{for }d\geq 0;$$
and the {\it Hilbert series} of $S/I$ by 
$$\HS_{S/I}(t)=\sum_{d\geq 0}\HF_{S/I}(d)t^d\in\bZ [\![ t ]\!].$$ 
It is easy to observe that the tangent space to the variety of powers at a point $[g^k]\in T_{[g_1^k]}\VP_{n,kd}^{(k)}$ is
$$
T_{[g^k]}\VP_{n,kd}^{(k)} = \{[g^{k-1}h] ~|~ h\in S_d\}.
$$
Therefore, under the genericity assumptions of Lemma \ref{Terracini}, we have that
$$
 T_P \sigma_s\left(\VP_{n,kd}^{(k)}\right) = \bP(I_{kd}) \subset \bP(S_{kd}),
$$
where $I_{kd}$ is the homogeneous part of degree $kd$ of the ideal $I = (g_1^{k-1},\ldots,g_s^{k-1})$, {where the $g_i$'s are general forms of degree $d$.} In particular, {the codimension of the secant variety is}
\begin{equation}\label{Terracini_codim}
 {\rm codim} ~\sigma_s\left(\VP_{n,kd}^{(k)}\right) = \HF_{S/I}(kd) - 1,
\end{equation}
and the generic $k$-rank  $\rk^\circ_k(d,n)$  is equal to the minimal $s$ such that $\HF_{S/I}(kd)=0$.

\subsection{Generic ranks of binary forms.} We consider now the case of binary forms. 

In \cite{Fr}, R.~Fr\"oberg observed that, for any ideal $I = (h_1,\ldots,h_s) \subset S$ with $\deg(h_i) = d_i$, for $i = 1,\ldots,s$, 
\begin{equation}\label{FHS}
\HS_{S/I}(t) \succeq \left\lceil \frac{\prod_{i=1}^s (1-t^{d_i})}{(1-t)^{n}}\right\rceil,
\end{equation}
where $\lceil\cdot\rceil$ {stands for} the truncation of the power series at the first non-positive coefficient and the {latter} inequality is in the lexicographic sense. An ideal {for which} \eqref{FHS} is an equality is called {\it Hilbert generic}. {\it Fr\"oberg's Conjecture} claims that generic ideals are Hilbert generic. We return to this in Appendix \ref{app:Froberg}. 

A crucial observation is that the property of being Hilbert generic is a Zariski open condition on the space of ideals in $S$ with a given number of generators of given degrees; see \cite{FL}. 
\begin{lemma}\label{lemma:Hilbert gen ideal}
 Let $g_1,\ldots,g_s$ be generic binary forms of degrees $d_1,\ldots,d_s$, respectively. Then, for any $s$-tuple of positive integers $k_1,\ldots,k_s$, the ideal $(g_1^{k_1},\ldots,g_s^{k_s})\subset S = \bC[x,y]$ is Hilbert generic. 
\end{lemma}
\begin{proof}
As we mentioned above, it suffices to prove the statement for a specialization of the $g_i$'s. Assume that they are powers of generic linear forms, i.e., $g_i = l_i^{d_i}$, for $i = 1,\ldots,s$. By \cite[Corollary 2.3]{GS}, the ideal $(l_1^{d_1k_1},\ldots,l_s^{d_sk_s})$ is Hilbert generic and we are done.
\end{proof}
Using Lemma~\ref{lemma:Hilbert gen ideal}, {we are able to settle Conjecture \ref{conj:otta} in the case of binary forms.}
\begin{theorem}\label{th:genkrank}
For $k,d \geq 2$, the generic $k$-rank of binary forms of degree $kd$ is
 $$\rk^\circ_k(2,kd)=\left\lceil\frac{{kd+1}}{d+1}\right\rceil.$$
\end{theorem} 
\begin{proof}
The case $k = 2$ is covered by \cite[Theorem 5]{FOS}. Assume that $k \geq 3$. By Lemma \ref{Terracini} and \eqref{Terracini_codim}, in order to compute the dimension of the $s$-th secant variety of the variety of $k$-th powers, we have to compute the Hilbert function {in degree $dk$} of the ideal $I=(g_1^{k-1},\ldots,g_s^{k-1})$, where the $g_i$'s are generic binary forms of degree $d$. By Lemma \ref{lemma:Hilbert gen ideal}, we know that $I$ is Hilbert generic. Therefore,
$\HS_{S/I}(t)=\left\lceil H(t)\right\rceil,$ where
$$
 H(t) = \frac{(1-t^{d(k-1)})^s}{(1-t)^2} =: \sum_{i\geq 0} H_it^i.
$$
{Since $k\geq 3$, for any $d(k-1) \leq i \leq dk$, we have} $H_i = (i+1) - s(i-d(k-1)+1)$. Then,
$$
 s_i \geq s_{i+1}, \text{ for all } i = d(k-1),\ldots,dk,
$$
{where we denote $s_i := \min\{ s ~|~ H_i \leq 0\}$.}
From this, it follows that, for $s \leq \left\lceil\frac{{kd+1}}{d+1}\right\rceil$, $$\HF_{S/I}(dk) = \max\{0,H_{kd}\} = \max\{0,(kd+1) - s(d+1)\}.$$

In particular, we conclude that $\rk_k^{\circ}(2,kd) = \left\lceil\frac{{kd+1}}{d+1}\right\rceil$. 
\end{proof}
As we mentioned in the introduction, the proof of Theorem \ref{th:genkrank} has been already obtained by different methods in \cite[Theorem 1.8]{Re1}. However, the approach used here is much more general and can extend to higher dimensions. We explain this idea in Appendix \ref{app:Froberg} where we show that a generalized version of Fr\"oberg's Conjecture implies Conjecture \ref{conj:otta} on generic $k$-ranks in any number of variables. In particular, we use this to settle the case of sums of squares in three variables, see Theorem \ref{thm:appendix_small_secants}.

\section{Maximal $3$-rank of binary sextics}
In this section {we settle Conjecture \ref{conj:maxkrank} in the case of sum of {cubes} decompositions of binary sextics.} 

\begin{theorem} Every binary sextic can be written as a sum of at most three cubes of binary quadratic forms.

\end{theorem}
\begin{proof}
We begin with an observation about binary cubics.
It has been known since the work of Sylvester that a binary cubic $h$
can be written as sum of two cubes of linear forms unless $h$ has a
square factor, and is 
not a cube; see \cite[{Theorem} 5.2]{Re1}. 

If $h$ is a cube,
it is trivially the sum of two cubes. If $h$ does not have a square
factor, then after a change of variables, $h(x,y) = xy(\al x + \be y)$
with $\al\be \neq 0$.  In this case, letting  $\om = e^{\frac{2\pi
    i}3}$, we have the identity
\[
3 \al \be (\om - \om^2) xy(\al x + \be y)=  (\om^2 \al x- \om \be y)^3
- (\om \al 
x - \om^2 \be  y)^3.
\]
Otherwise, after a change of variables,  $h(x,y) = x^2y$ and
\[
6x^2y = (-x+y)^3 - 2y^3 + (x+y)^3.
\]
Thus, in every case, a binary cubic is a sum of at most three cubes of
linear forms. {Sylvester's} algorithm allows
one to write a cubic as a sum of cubes without factoring it, {see Example \ref{ex: sylvester}.}

Consider now the general binary sextic
\[
p(x,y) = \sum_{k=0}^6 \binom 6k a_k x^{6-k}y^k 
, \quad a_k \in \mathbb C.
\]
If $p=0$, then there is nothing to prove. Otherwise, 
we may make an invertible linear change of
variables, after which $p(1,0)  = a_0\neq 0$. 
{If we set
\[
\begin{gathered}
q(x,y) = x^2 + \left( \frac{2a_1}{a_0}\right) xy +
\left(\frac{5a_0a_2-4a_1^2}{a_0^2}\right) y^2,
\end{gathered}
\]
then,} we have the expression
\begin{equation}\label{E:1}
p(x,y) = a_0q(x,y)^3 + \frac 1{a_0^5} y^3c(x,y),
\end{equation}
where
$c(x,y) = c_0 x^3 + 3 c_1 x^2y + 3c_2 xy^2 + c_3y^3$, with
\begin{align*}
c_0 &= 20a_0^3(2a_1^3 - 3a_0a_1a_3+a_0^2 a_3), \\
c_1 &= 5a_0^2(4 a_1^2 a_2 - 5 a_0 a_2^2 + a_0^2 a_4), \\
c_2 &= 2a_0(-16 a_1^5 + 40 a_0 a_1^3 a_2 - 25 a_0^2 a_1 a_2^2 + a_0^4
a_5),\\
c_3 &= 64 a_1^6 - 240 a_0 a_1^4 a_2 + 300 a_0^2 a_1^2 a_2^2 - 125 a_0^3 a_2^3 + 
 a_0^5 a_6.
\end{align*}

If the discriminant {$\Delta(c)$ of $c$  is {non-vanishing},} then $c(x,y)$ has distinct
factors and, by Sylvester, it is a sum of two cubes of linear forms.
Thus \eqref{E:1} gives $p$  as a sum of three
cubes of quadratic forms. A computation shows that 
{
\[
\Delta(c) = 
 -540a_0^6D(p),
\]
}
where $D(p)$ is a polynomial in the $a_k$'s of degree 18 with 128 terms.  It follows that $p$
 is a sum of three cubes of linear forms unless $D(p) = 0$. This was known for general sextics by \cite[{Corollary} 4.2]{Re2}.

Suppose $D(p) = 0$. If $c$ is the cube of a linear form, then
\eqref{E:1} gives $p$ as a sum of two cubes. In the remaining case,
$c$ is a cubic with a square factor. As a first step, we
may rewrite $p$ as 
\[\label{E:2}
p(x,y) = (ax^2 + 2b x y + c y^2)^3 + y^3(r x + s y)^2(t x + u y),
\]
where $ru - st \neq 0$. 
If $a = 0$, then $p(x,y) = y^3h(x,y)$ for a cubic $h$, and (as noted earlier)
 $h$ must be a sum of three cubes of linear forms, so $p$ is then a
 sum of three cubes. We may therefore assume 
that $a \neq 0$ and distinguish two cases.

First, suppose $r=0$ in \eqref{E:2}. Then $st \neq 0$. We may scale $y$ so that
$s=1$. Then, since $t\neq 0$, we make a further invertible change of variables
$(tx+uy,x) \mapsto (x,y)$. This case then reduces to 
\[
p^{(1)}(x,y) = (ax^2 + 2b x y + c y^2)^3 + xy^5.
\]
The same argument as before shows that we may assume $a\neq 0$ for $p^{(1)}$.

Otherwise, $r \neq 0$, so we make the change of variables $(rx +
sy,y) \mapsto (x.y)$, so that $y^3(r x + s y)^2(t x + u y) \mapsto
y^3x^2(tx + uy)$, where $u \neq 0$, and we may again scale $y$ so that
$u=1$. This gives the second case
\[
p^{(2)}(x,y) =  (ax^2 + 2b x y + c y^2)^3 + x^2y^3(t x + y).
\]

For either case $p^{(j)}$ and $T \in \mathbb C$, let $p^{(j)}_T(x,y) =
p^{(j)}(x,Tx + y)$, so 
$p^{(j)}(x,y) = p^{(j)}_T(x, -Tx+y)$. It will suffice to write
$p^{(j)}_T$ as a sum of three cubes. Write 
\[
p^{(j)}_T(x,y) = \sum_{k=0}^6 \binom 6k a_k(T)x^{6-k}y^k.
\]
Here, $a_k(T)$ is a polynomial in $T$ of degree $6-k$.  There are at
most 6 values of $T$ which must be avoided to ensure that $a_0(T) \neq 0$. 

We now compute $D(p^{(j)}_T)$, with the intent of finding a
value of $T$ for which $D(p_T)\neq 0$ . It turns out that $D(p^{(j)}_T)$ is a
massive polynomial of degree 72 in the coefficients of $p^{(j)}$. If
this polynomial 
is non-zero, then except for 
at most $6+72$ values of $T$, $p_T$ (and hence $p$) is a sum of three
cubes. Thus, our situation reduces to deriving a contradiction from
the assumption that $D(p^{(j)}_T)$ is the zero polynomial. (For
computational reasons, this is why we reduced to $p^{(j)}$
above.)    This is easiest to see by
considering the lowest order term in various special cases.

A computation shows that the lowest order term of  $D(p^{(1)}_T)$ 
 is $\frac 1{36} a^{42} T^2$. Since $a \neq 0$, it follows that this
 is not the zero polynomial and so for all but finitely many values of
 $T$, $p^{(1)}_T$ is a sum of three cubes. 

The computation of $D(p^{(2)}_T)$ is  trickier, and we divide
it into four cases. If $c,t \neq 0$, the lowest order term of
$D(p^{(2)}_T)$ is  $-\frac 1{45} a^{40}c^2t T$, and since  $a,c,t \neq
0$, this term is not zero. If $t=0, c\neq 0$, the lowest order term is
$-\frac 2{45} a^{40}c^2 T^2 \neq 0$.
If $c=0,t \neq 0$, the lowest order term is $ -\frac 1{135}
a^{36}t^3 T^3 \neq 0$. If $t=c=0$, the lowest order term of
$D(T)$ is  $-\frac 8{135}
a^6 T^6 \neq 0$. This completes the proof.
\end{proof}

\begin{example} In practice, this algorithm leads to some nasty
  expressions. We start with a specially-cooked simple one. Suppose
\[
p(x,y) = x^6 + 3 x^5 y - 3 x^4 y^2 - 11 x^3 y^3 + 9 x^2 y^4 + 21 x y^5 - y^6.
\]
Then, 
\[
p(x,y) = (x^2 + x y - 2y^2)^3 + y^3(3x^2y + 9xy^2 + 7 y^3).
\]
Applying Sylvester's algorithm to $h(x,y) = 3x^2y + 9xy^2 + 7 y^3$ = $0\cdot
y^3 + 1 \cdot 3x^2y + 3 \cdot 3xy^2 + 7 \cdot y^3$, we see that 
\[
\begin{pmatrix}
0 & 1 & 3\\
1 & 3 & 7\\ 
\end{pmatrix}
\cdot
\begin{pmatrix}
2 \\ -3 \\ 1
\end{pmatrix}
=\begin{pmatrix}
0\\0
\end{pmatrix} 
\]
and $2x^2 - 3xy + y^2 = (x - y)(2x - y)$, so there exist $\la_k$ so
that $h(x,y) = \la_1 (x + y)^3 + \la_2(x + 2y)^3$. Indeed, $\la_1 =
-1$ and $\la_2 =1$ and
\[
p(x,y) = (x^2 + x y - 2y^2)^3 + y^3(x + 2 y)^3 - y^3 (x + y)^3.
\]
\end{example}
\begin{example}
For a less trivial example, 
suppose $p(x,y) =  x^6 + x^5y+x^4y^2+x^3y^3+x^2y^4$ $+xy^5+y^6$.
Then
\begin{equation*}
\begin{gathered}
p(x,y) - \left(x^2 + \tfrac 13 x y + \tfrac 29 y^2 \right)^3 = \\
\frac 7{729} y^3 (54 x^3 + 81 x^2 y + 99 x y^2 + 103 y^3).
\end{gathered}
\end{equation*}
An application of Sylvester's algorithm shows that 
\begin{equation*}
\begin{gathered}
54 x^3 + 81 x^2 y + 99 x y^2 + 103 y^3 = \\
m_1 (78 x + (173 - \sqrt{20153}) y)^3 +  
 m_2 (78 x + (173 + \sqrt{20153}) y)^3, \\
m_1 = \frac{20153 + 134 \sqrt{20153}}{354209128},\qquad m_2 =  \frac{20153 -
  134 \sqrt{20153}}{354209128} 
\end{gathered}
\end{equation*}
This gives a simple sextic $p$ as a sum of three cubes in an ugly way
and gives no hint about the existence of the formula 
\[
p(x,y) = \sum_{\pm} \left( \tfrac{9 \pm \sqrt{-3}}{18}\right)(x^2 +
\tfrac{1\pm\sqrt{-3}}2 xy + y^2)^3. 
\]
\end{example}
\begin{example}
For another example, set
\[
p(x,y) = x^6 + 3x^5 y + y^6.
\]
We have 
\[
p(x,y) - (x^2 + x y - y^2)^3 = y^3(5x^3 - 3x y^2 + 2y^3)
\]
and yet another application of Sylvester's algorithm gives
\[
5x^3 - 3x y^2 + 2y^3 = \tfrac{20-9\sqrt{5}}{20}((-5-2\sqrt 5)x + y)^3 +
\tfrac{20+9\sqrt{5}}{20}((-5+2\sqrt 5)x + y)^3.
\]

On the other hand, if $p(x,y) = x^6 + 3x y^5 + y^6$,
then the first step of the algorithm leaves us with $c(x,y) = y^2(3x+y)$,
and we must invoke $p_T$. A computation shows that $D(T)$ equals $T^2$
times a polynomial which is irreducible over $\mathbb Q$. We take
$T=-1$, and
\[
p_{-1}(x,y) =  x^6 + 3x (-x+y)^5 + (-x+y)^6 = -x^6 + 9 x^5 y - 15 x^4
y^2 + 10 x^3 y^3 - 3 x y^5 + y^6.
\]
Following the algorithm,
\[
p_1(x,y) + (x^2 - 3 x y - 4y^2)^3 = y^3(55 x^3 - 60 x^2 y - 147 x y^2 - 63 y^3)
\]
and a further invocation of Sylvester, followed by the reversed change
of variables yields 
\[
\begin{gathered}
x^6 + 3 x y^5 + y^6 =\\ (6 x^2 + 11 x y + 4 y^2)^3 + \la_+(\al_+ x^2 +
\be x y +  y^2)^3 +  \la_-(\al_- x^2 +
\be x y +  y^2)^3  \\
 \la_{\pm} = \frac{4445 \pm \sqrt{5632445}}{2282}, \quad
  \al_{\pm} = \frac{6727 \pm \sqrt{5632445}}{2282}, \be_{\pm} =
  \frac{9009 \pm \sqrt{5632445}}{326}.
\end{gathered}
\]
\end{example}
An alternative approach is to observe that for a sextic $p$, there is
usually a quadratic $q$ so that $p - q^3$ is even{, i.e., is a cubic in $\{x^2,y^2\}$}. It is enough to look at the
coefficients of $x^5y, x^3y^3, xy^5$ and solve the equations for the
coefficients of $q$. So, it is usually a sum of two cubes of even quadratic forms. If this doesn't
work, apply it to $p_T$.


\section{Computing $k$-ranks}\label{sec: computing ranks}
In this section, we propose a {procedure} to compute $k$-th Waring decompositions based on a geometric description of the varieties of powers as an explicit linear projection of Veronese varieties. 

The variety of powers $\VP_{n,kd}^{(k)}$ is the image of the regular map
$$
 \nu_{n,kd}^{(k)} : \bP(S_d) \rightarrow \bP(S_{kd}),~~ [g] \mapsto [g^k].
$$
Denote by $B_i$ the set of multi-indices $\al = (\al_1,\ldots,\al_n) \in \bN^n$ and $|\al| = \al_1 + \ldots + \al_n = d$. 

{Let $T$ be the coordinate ring of $\bP(S_d)$ whose variables are labelled by $B_d$, i.e., $T = \bC[Y_\al ~:~ \al \in B_d]$. In particular, if $S = \bC[x_1,\ldots,x_n]$, we identify $T_1$ with $S_d$ by setting $Y_\al = x^\al := x_1^{\al_1}\cdots x_n^{\al_n}$.} Given $g \in S_d$, we denote by $l_{g}$ the corresponding linear form in $T_1$. Thus, we consider the usual Veronese embedding
$$
 \nu_{N,k} : \bP(T_1) \rightarrow \bP(T_{k}),~~ [l] \mapsto [l^k],
$$
whose image is the Veronese variety ${\rm Ver}_{N,k}$, where $N = {d+n-1 \choose n-1} - 1$.

The substitution $Y_\al \mapsto x^\al$ gives a linear projection 
$
 \pi_{n,kd}^{(k)} : \bP(T_k) \rightarrow \bP(S_{kd})
$
and we get the diagram 
$$
\xymatrix{
\bP(T_1)\ni [l_g] \ar[rr]^{\nu_{N,k}} \ar@{=}[d] & & [l_g^k]\in {\rm Ver}_{N,k}\subset\bP(T_k) \ar@<1ex>@{-->}[d]^{\pi_{n,kd}^{(k)} }\\
\bP(S_d)\ni [g] \ar[rr]^{\nu_{n,kd}^{(k)}}& &  [g^k]\in \VP_{n,d}^{(k)}\subset\bP(S_{kd}) \\
}
$$
The center of the projection $\pi_{n,kd}^{(k)} $ has an explicit interpretation in terms of the ideal defining a specific Veronese variety. Recall that Veronese varieties are determinantal, see \cite{Pu98}. In particular, the Veronese variety ${\rm Ver}_{n,d}\subset \bP(S_{d})$ is generated by the $2\times 2$ minors of the {\it $i$-th catalecticant matrix}
$$
 {\rm cat}_i(n,d) = \left(Y_{\be_1+\be_2}\right)_{\be_1,\be_2} \in {\rm Mat}^{B_i \times B_{d-i}}, \text{ for any } i = 1,\ldots,d-1,
$$
where $Y_{\al}$ are coordinates of the space $\bP(S_d)$ corresponding to the standard monomial basis of $S_d$.
\begin{lemma}\label{lemma:center}
 In the same notation as above, the center of the linear projection $\pi_{n,kd}^{(k)}$ is
 $$
  E = \bP([I_{n,d}]_k),
 $$
 where $[I_{n,d}]_k$ is the $k$-th homogeneous part of the ideal $I_{n,d} \subset T$ defining the Veronese variety ${\rm Ver}_{n,d}$.
\end{lemma} 
\begin{proof}
 The center of the projection is given by the linear space of homogeneous polynomials $f\in T_k$ annihilated by the substitution $Y_{\alpha} \mapsto x^\alpha$, for any $\alpha \in B_d$. This is, by definition, the set of forms of degree $k$ vanishing on the Veronese variety ${\rm Ver}_{n,d}$.
\end{proof}
It is easy to see that the Veronese variety ${\rm Ver}_{N,k}\subset \bP(T_k)$ does not intersect the center of the projection. Therefore, the projection $\pi_{n,kd}^{(k)}$ is a regular map {restricted to} ${\rm Ver}_{N,k}$ and it maps the Veronese variety onto the variety of powers $\VP^{(k)}_{n,d}$. Hence, by \cite[Theorem 7, Section I.5.3]{Sh94}, we have that the restriction $\pi_{n,kd}^{(k)}  : {\rm Ver}_{N,k} \rightarrow \VP^{(k)}_{n,d}$ is a finite surjective map.

{Now, to find a $k$-th Waring decomposition {of} a given form $f$ of degree $kd$, we want to project {using} $\pi_{n,kd}^{(k)}$ a classical Waring decomposition in $T_k$ of some element in the fiber over $f$.} A similar idea has been also used in \cite{BB}.

Given $f \in S_{kd}$, we denote by $\cF_f$ the fiber $(\pi_{n,kd}^{(k)})^{-1}([f])$. In particular, if $f_0\in T_k$ is any element such that $\pi_{n,kd}^{(k)}([f_0]) = [f]$, then, we have $\cF_{f} = \left\langle \{[f_0]\} \cup E\right\rangle\subset \bP(T_k)$.
\begin{lemma}\label{lemma:rank_fiber}
Let $f\in S_{kd}$. Then, 
 $$
  \rk_k(f) = \min\left\{\rk(g) ~|~ [g] \in \cF_f\smallsetminus E\subset \bP(T_k)\right\}.
 $$
\end{lemma}
\begin{proof}
 Given a Waring decomposition $g = \sum_{i=1}^r l_i^k$ of an element $g \in \cF_f\smallsetminus E$, by substituting the $Y$'s with the corresponding monomials in the $x$'s, we obtain a Waring decomposition of $f$ as sum of $k$-th powers. Therefore, $\rk_k(f) \leq \min\left\{\rk(g) ~|~ [g] \in \cF_f\smallsetminus E\right\}$.
 
Conversely, let $f = \sum_{i=1}^r g_i^k$ be a minimal $k$-th Waring decomposition {representing} $f$, where $g_i = \sum_{|\alpha| = d/k} c_{i,\alpha} x^\alpha$, for $i = 1,\ldots,r$. Set $l_{g_i} := \sum_{|\alpha| = d/k} c_{i,\alpha} Y_\alpha \in T_1$. Then, $\sum_{i=1}^r l_{g_i}^k$ is a sum of $k$-th powers of linear forms of an element in $\cF_f\smallsetminus E$ since its image with respect to the linear projection coincides with $[f]$. Therefore, $\rk_k(f) \geq \min\left\{\rk(g) ~|~ [g] \in \cF_f\smallsetminus E\right\}$. 
\end{proof}

Hence, we reduce our problem of computing a $k$-th Waring decomposition of $f$ to the problem of minimizing the classical Waring rank in $\cF_f\setminus E$. Several algorithms for computations of Waring ranks and Waring decompositions have been proposed in the literature. Already, J.J.~Sylvester \cite{Syl1, Syl2} considered the case of binary forms, see \cite{CS} for a modern exposition. Other solutions can be found, for example, in \cite{BCMT10,BGI11,OO13}, but they all work efficiently only under some additional constraints on the considered polynomial. {However, due to Lemma \ref{lemma:rank_fiber}, we {can also use these algorithms to study} $k$-th Waring ranks.}

As an illustration, we apply Sylvester's {\it catalecticant method} to compute $k$-th ranks of binary forms.

\subsection{Sylvester's catalecticant method.}
\noindent {Let us recall the basic notions of Apolarity Theory.}

Let $S = \bC[x_1,\ldots,x_n]$ and $R = \bC[y_1,\ldots,y_n]$ be standard graded polynomial rings. For $i \leq 0$, we set $S_i$ and $R_i$ to be $0$. 
We consider the {\it apolar action} $\circ$ where polynomials in $R$ act over $S$ as partial differentials and, for any polynomial $f \in S_d$, we define the {\it apolar ideal} as $f^\perp = \left\{g\in R ~|~ g\circ f = 0\right\}$. More explicitly, the homogeneous part in degree $i$ of $f^\perp$ is given by the kernel of the {\it $i$-th catalecticant matrix of $f$}
$$
 {\rm cat}_i(f) : R_i \rightarrow S_{d-i},~~  y^\alpha \mapsto \frac{\partial^if}{\partial_{x_1}^{\alpha_1}\cdots \partial_{x_n}^{\alpha_n}}.
$$
{The key lemma which relates classical Waring decompositions to ideals of reduced points is as follows.}
\begin{lemma}[{\rm Apolarity Lemma \cite[Lemma 1.15]{IK}}]\label{lemma:apolarity}
 Let $f \in S$. Then, the following are equivalent:
 \begin{enumerate}
  \item $f$ has a decomposition of length $s$ as sum of powers of linear forms;
  \item $f^\perp$ contains an ideal defining a set of reduced points of cardinality $s$.
 \end{enumerate}
\end{lemma}

\begin{example}[{\it Sylvester's algorithm}]\label{ex: sylvester}
Given a binary form $f$, compute the apolar ideal $f^\perp$. Since any apolar ideal $f^\perp$ is artinian and Gorenstein, in the case of two variables, it is also a complete intersection. In particular, $f^\perp = (g_1,g_2)$ where $\deg(g_1)+\deg(g_2) = \deg(f) + 2$. 

If $g_1$ is square-free, then the ideal $(g_1)$ defines a set of reduced points in $\bP^1$ and then, by Apolarity Lemma, we conclude that the Waring rank of $f$ is equal to $\deg(g_1)$. Otherwise, it is possible to find a square-free element in $f^\perp$ of degree $\deg(g_2)$ and we conclude that the Waring rank of $f$ is $\deg(g_2)$.
\end{example}

{Two examples below illustrate} how we can use this method to study $k$-th Waring ranks.

\begin{example}[Upper bound]
Here, we show how to compute a $k$-th Waring decomposition and obtain upper bounds on $k$-th ranks. Let $f = x^{6} y^{2}-x^{3} y^{5}+x^{2} y^{6}-x y^{7}$ be a binary octic. 

We pick an element in the fiber, for example $f_0 = {Y}_{0}^{3} {Y}_{2}-{Y}_{0} {Y}_{1} {Y}_{2}^{2}+{Y}_{0} {Y}_{2}^{3}-{Y}_{1} {Y}_{2}^{3}$. In this case, we have 
$$
	\left(f_0^\perp\right)_2 = \left\langle Y_1^2, Y_0^2 - 9Y_0Y_1 + 3Y_1Y_2 \right\rangle.
$$
This defines a non-reduced $0$-dimensional scheme, hence, we have that the rank of $f_0$ is at least $5$. Hence, we look for another element in the fiber having rank $4$. 

The general element of the fiber $\cF_f \setminus E$ is
\begin{align*}
 F_c = \lambda f_0 - (Y_{20}Y_{02}-Y_{11}^2) (c_{0}Y_{20}^2 + c_{1}Y_{20}Y_{11} + c_{2}Y_{20}Y_{02} + c_{3}Y_{11}^2 +c_{4}Y_{11}Y_{02} + c_{5}Y_{02}^2),
\end{align*}
{where $\lambda \in \bC$ and $c = (c_0,\ldots,c_5) \in \bC^6$. Since we want $F_c \not\in E$, we may assume $\lambda = 1$.} Now, we consider the apolar action of $\bC[y_{20},y_{11},y_{02}]$ on $\bC[Y_{20},Y_{11},Y_{02}]$, and we get the $2$-nd catalecticant matrix
$${\rm cat}_2(F_c) = \bgroup\begin{pmatrix}0&
       0&
       6 {c}_{0}+6&
       {-4 {c}_{0}}&
       2 {c}_{1}&
       4 {c}_{2}\\
       0&
       {-4 {c}_{0}}&
       2 {c}_{1}&
       {-6 {c}_{1}}&
       -2 {c}_{2}+2 {c}_{3}&
       2 {c}_{4}-2\\
       6 {c}_{0}+6&
       2 {c}_{1}&
       4 {c}_{2}&
       -2 {c}_{2}+2 {c}_{3}&
       2 {c}_{4}-2&
       6 {c}_{5}+6\\
       {-4 {c}_{0}}&
       {-6 {c}_{1}}&
       -2 {c}_{2}+2 {c}_{3}&
       {-24 {c}_{3}}&
       {-6 {c}_{4}}&
       {-4 {c}_{5}}\\
       2 {c}_{1}&
       -2 {c}_{2}+2 {c}_{3}&
       2 {c}_{4}-2&
       {-6 {c}_{4}}&
       {-4 {c}_{5}}&
       {-6}\\
       4 {c}_{2}&
       2 {c}_{4}-2&
       6 {c}_{5}+6&
       {-4 {c}_{5}}&
       {-6}&
       0\\
       \end{pmatrix}\egroup.$$ 
       {Now, we notice that a point vanishing all the $5\times 5$ minors is
      $$\overline{c} = (0,0,0,1,0) \in \bC^6,$$}
      i.e., $F_{\overline{c}} = f - (Y_1^3Y_2-Y_0Y_1Y_2^2) = {Y}_{0}^{3} {Y}_{2}+{Y}_{0} {Y}_{2}^{3}-Y_1^3Y_2-{Y}_{1} {Y}_{2}^{3}$.
      Then, the kernel of ${\rm cat}_2(F_{\overline{c}})$ is given by 
      $$
      \ker ({\rm cat}_2(F_{\overline{c}})) = \langle Q_1,Q_2 \rangle = \left\langle Y_0Y_1, Y_0^2 + Y_1^2 - Y_2^2 \right\rangle.
      $$
      The ideal $I = (Q_1,Q_2)$ defines a set of four points; in particular,
      $$\bX = \{(0:1:1),~(0:1:-1),~(1:0:1),~(1:0:-1)\}.$$
      By solving an easy linear system, we obtain the following expression \footnote{A {\it Macaulay2} script describing this example can be found in the personal webpage of the second author at \href{https://sites.google.com/view/alessandrooneto/research/list-of-papers}{https://sites.google.com/view/alessandrooneto/research/list-of-papers}.}
      $$
      	8f = (xy-y^2)^4 - (x^2-y^2)^4 + (x^2+y^2)^4 - (xy+y^2)^4.
      $$
\end{example}
\begin{remark}\label{rmk: HF of points}
Recall the following important properties on the Hilbert function of any homogeneous ideal $I_{\bX}\subset S$ defining a set of reduced points in projective space and which is contained in the apolar ideal $f^\perp$ {of any given polynomial $f$:}
 \begin{enumerate}
  \item $\HF_{R/I_{\bX}}$ is strictly increasing until it becomes constant and equal to the cardinality $|\bX|$;
  \item $\HF_{R/I_{\bX}}(i) \geq \HF_{R/f^\perp}(i)$, for all $i \geq 0$.
 \end{enumerate}
Hence, since the Hilbert function $\HF_{R/f^\perp}(i)$ is equal to the dimension of the image of ${\rm cat}_i(f)$ which is the rank of the matrix, we have that the rank of a catalecticant matrix ${\rm cat}_i(f)$ is always a lower bound for the Waring rank of $f$, for any $i = 1,\ldots,\deg(f)$.
 \end{remark}

\begin{example}[Lower bound]\label{example:xy7} Using the latter remark, we can also find lower bounds on $k$-th ranks. We consider the binary octic {$m = xy^7 \in S_8$. The monomial $M_0 = Y_{11}Y_{02}^3$ satisfies $\pi^{(4)}_{2,2}([M_0]) = [m]$} and it is well-known that its Waring rank is $\rk(M_0) = 4$. Therefore, $\rk_4(m) \leq 4$. We use Lemma \ref{lemma:rank_fiber} to prove that the equality holds. A general element in $\cF_m \setminus E$ is given by
 \begin{align*}
 F_c = \lambda (Y_{11}& Y_{02}^3)  \\
 & - (Y_{20}Y_{02}-Y_{11}^2) (c_{0}Y_{20}^2 + c_{1}Y_{20}Y_{11} + c_{2}Y_{20}Y_{02} + c_{3}Y_{11}^2 +c_{4}Y_{11}Y_{02} + c_{5}Y_{02}^2),
 \end{align*}
{where $\lambda \in \bC$ and $c = (c_0,\ldots,c_5) \in \bC^6$. Since we want $F_c \not\in E$, we may assume $\lambda = 1$.} We need to prove that, for {\it any} choice of the $c_i$'s, the Waring rank of $F_c$ is at least $4$. Since {$M_0$} has rank equal to $4$, we can restrict to the case where not all the $c_\al$'s are equal to zero. Now, the $2$-nd catalecticant matrix of $F_c$ is 
 $${\rm cat}_2(F_c) = \begin{pmatrix}0&
      0&
      6 {c}_{0}&
      {-4 {c}_{0}}&
      2 {c}_{1}&
      4 {c}_{2}\\
      0&
      {-4 {c}_{0}}&
      2 {c}_{1}&
      {-6 {c}_{1}}&
      -2 {c}_{2}+2 {c}_{3}&
      2 {c}_{4}\\
      6 {c}_{0}&
      2 {c}_{1}&
      4 {c}_{2}&
      -2 {c}_{2}+2 {c}_{3}&
      2 {c}_{4}&
      6 {c}_{5}\\
      {-4 {c}_{0}}&
      {-6 {c}_{1}}&
      -2 {c}_{2}+2 {c}_{3}&
      {-24 {c}_{3}}&
      {-6 {c}_{4}}&
      {-4 {c}_{5}}\\
      2 {c}_{1}&
      -2 {c}_{2}+2 {c}_{3}&
      2 {c}_{4}&
      {-6 {c}_{4}}&
      {-4 {c}_{5}}&
      6\\
      4 {c}_{2}&
      2 {c}_{4}&
      6 {c}_{5}&
      {-4 {c}_{5}}&
      6&
      0\\
      \end{pmatrix}$$
 Computing the {radical of the ideal generated} by {the} $4\times 4$ minors with the algebra software {\it Macaulay~2} \cite{M2} {gives the ideal} $(c_0,c_1,c_2,c_3,c_4)$. {Therefore, if $c_i \neq 0$, for some $0 \leq i \leq 4$, we have that the rank of ${\rm cat}_2(F_c)$ is at least $4$ and, by Remark \ref{rmk: HF of points}, the Waring rank of $F_c$ is at least $4$.
 \\ Now, if we assume that $c_0 = \ldots = c_4 = 0$ and also $c_5 = 0$, we have that $F_c = Y_{11}Y_{02}^3$ which has rank $4$. If $c_5 \neq 0$, we are left with the matrix
 $$
 {\rm cat}_2(F_c) = \begin{pmatrix}0&
      0&
      0&
      0&
      0&
      0\\
      0&
      0&
      0&
      0&
      0&
      0\\
      0&
      0&
      0&
      0&
      0&
      6 {c}_{5}\\
      0&
      0&
      0&
      0&
      0&
      {-4 {c}_{5}}\\
      0&
      0&
      0&
      0&
      {-4 {c}_{5}}&
      6\\
      0&
      0&
      6 {c}_{5}&
      {-4 {c}_{5}}&
      6&
      0\\
      \end{pmatrix}
 $$
 which has rank $3$ and the kernel is given by 
 $$
      \ker ({\rm cat}_2(F_{{c}})) = \langle Q_1,Q_2,Q_3 \rangle = 
      \left\langle y_{20}^2,~y_{20}y_{11},~2y_{11}y_{02}+3y_{11}^2\right\rangle.
  $$
  The ideal $(Q_1,Q_2,Q_3)$ defines a curvilinear $0$-dimensional scheme of length $3$ with support at the point $(0:0:1)\in\bP^2$. {Therefore}, as explained in \cite[Theorem 4]{BGI11}, the rank of $F_c$ is $7$.
  \\ Hence, all $F_c \in \cF_m \setminus E$ have rank at least $4$ and, by Lemma \ref{lemma:rank_fiber}, $xy^7$ has $4$-th rank at least $4$. Since, as we noticed, $xy^7 = (xy^3)y^4$ and the monomial $xy^3$ has Waring rank $4$, we conclude that $\rk_4(xy^7) = 4$. }\footnote{A {\it Macaulay2} script describing this example can be found in the personal webpage of the second author at \href{https://sites.google.com/view/alessandrooneto/research/list-of-papers}{https://sites.google.com/view/alessandrooneto/research/list-of-papers}.}
\end{example}

 In \cite{BCMT10,BGI11,OO13}, the authors give other linear algebra methods to compute Waring ranks and Waring decompositions. These methods can be used in our procedure instead of the classical catalecticant method to find better bounds on $k$-th Waring ranks. However, we have to observe that these methods are not always effective and might fail. Also, if we consider polynomials with higher number of variables or higher degrees, the dimension of the fiber grows quickly and this can make the computation very difficult and heavy. 
 
\appendix
\section{Fr\"oberg's Conjecture and generic ranks}\label{app:Froberg}
In \cite{Fr}, R.~Fr\"oberg studied Hilbert series of generic ideals.

\begin{conjecture}{\rm [Fr\"oberg's Conjecture, \cite{Fr}]}\label{conj:B} Generic ideals are Hilbert generic, i.e., for $I =(g_1,\ldots,g_s)$ where $g_1,\ldots,g_s$ are generic forms of degrees $d_1,\ldots,d_s$ in $n$ variables, then,
\begin{equation}\label{FC_formula}
 \HS_{S/I}(t) = \left\lceil\frac{\prod_{i=1}^s(1-t^{d_i})}{(1-t)^{n}}\right\rceil \in \bZ[\![t]\!],
\end{equation}
where $\left\lceil \cdot \right\rceil$ is the truncation of the power series at its first non-positive coefficient.
\end{conjecture}
\noindent Fr\"oberg's Conjecture attracted the {attention of many mathematicians, but, at the moment it is known to be true only in the following cases. For} $s \leq n$ (easy exercise), $s = n+1$ (by R. Stanley, \cite{St78}), $n = 2$ (by R. Fr\"oberg, \cite{Fr}) and $n = 3$ (by D. Anick, \cite{An}). More recently, G. Nenashev settled a large number of cases \cite{Ne17}. In \cite{N}, L.~Nicklasson suggested the following more general version that allow us to directly relate this commutative algebra problem to our computations on generic $k$-ranks.
\begin{conjecture}{\rm \cite{N}}\label{conj:C}
 Let $k \geq 2$. Let $I = (g_1^{k-1},\ldots,g_s^{k-1})$ be an ideal generated by $(k-1)$-th powers of generic homogeneous polynomials of degree $d > 1$. Then, $I$ is Hilbert generic.
\end{conjecture}

By using the same approach as in our proof of Theorem \ref{th:genkrank}, we prove the following.
\begin{theorem}\label{thm:appendix_Froberg}
 Conjecture \ref{conj:C} implies Conjecture \ref{conj:otta}.
\end{theorem}
First, we need the following lemma.
\begin{lemma}\label{lemma:1}
 Let $d,k,n \geq 2$ be positive integers and $H(t) = \frac{(1-t^{d(k-1)})^s}{(1-t)^{n}} =: \sum_{i\geq 0} H_it^i$. {Then, $$s_{kd} \leq s_{kd-1} \leq \ldots \leq s_{d(k-1)},$$ where $s_i := \min\{ s ~|~ H_i \leq 0\}$.}
\end{lemma}
\begin{proof}
{For $n = 2$, the claim is part of the proof of Theorem \ref{th:genkrank}. Then, we assume $n \geq 3$.} 

For $k \geq 3$, the statement directly follows from the fact that 
{$$H_i = {i+n-1 \choose n-1} - s {i-d(k-1)+n-1 \choose n-1},$$} for any ${d(k-1) \leq i \leq dk}$. For $k = 2$, the same formula holds for {$i < dk$}, but we have 
{$$H_{2d} = {2d+n-1 \choose n-1} - s {d+n-1 \choose n-1} + {s \choose 2}.$$} Hence, we still have to prove that, for $k = 2$, we have also $s_{2d} \leq s_{2d-1}$.

Since {$s_{2d-1} = \left\lceil\frac{{2d+n-2 \choose n-1}}{{n+d-2 \choose n-1}}\right\rceil$}, it is enough to prove that
$$
 H_{2d} \leq 0, \text{ for } s = \frac{{2d+n-2 \choose n-1}}{{n+d-2 \choose n-1}}.
$$
Then, we need to {show} that
 \begin{align*}
  2{{n+d-2} \choose {n-1}}^2{{n+2d-1} \choose {n-1}} & - 2{n+2d-2 \choose n-1}{n+d-1 \choose n-1}{n+d-2 \choose n-1} + \\
  & + {n+2d-2 \choose n-1}^2 - {n+2d-2 \choose n-1}{n+d-2 \choose n-1} \leq 0.
 \end{align*}
 It is easy to see that ${n+d-2 \choose n-1}{n+2d-1 \choose n-1} -  {n+2d-1 \choose n-1}{n+d-2 \choose n-1} = - {n+d-2 \choose n-1}{n+2d-2 \choose n-2}$; hence, we need to show that
  \begin{equation}\label{equation:Frob_implies_conjA_2}
 {n+2d-2 \choose n-1}^2 - {n+2d-2 \choose n-1}{n+d-2 \choose n-1} \leq 2{n+d-2 \choose n-1}^2{n+2d-2 \choose n-2}.
 \end{equation}
 For $n = 3$ and $d = 2$, we can directly check that this {inequality holds}. For $n \geq 3$, $d \geq 2$ and $nd > 6$, we show the following inequality, stronger than \eqref{equation:Frob_implies_conjA_2},
 \begin{equation}\label{equation:Frob_implies_conjA_3}
 {n+2d-2 \choose n-1}^2  \leq 2{n+d-2 \choose n-1}^2{n+2d-2 \choose n-2}
 \end{equation}
 By expanding the binomials in \eqref{equation:Frob_implies_conjA_3}, we have
 \begin{equation}
  \frac{\big((n+2d-2)!\big)^2}{\big((n-1)!\big)^2\big((2d-1)!\big)^2} \leq 2 \frac{\big((n+d-2)!\big)^2}{\big((n-1)!\big)^2\big((d-1)!\big)^2}\cdot \frac{\big((n+2d-2)!\big)}{\big((n-2)!\big)\big((2d)!\big)}.
 \end{equation}
 Simplifying the latter inequality, we are left with
 \begin{equation}\label{equation:Frob_implies_conjA_4}
  \prod_{i=1}^{n-2} (2d+i) \leq \frac{1}{2} \cdot \prod_{i=1}^{n-2} \frac{\left(d+i\right)^2}{i}.
 \end{equation}
{For $n = 3$, we need to show that 
 $$
 2d+1 \leq \frac{1}{2}\left(d+1\right)^2 \Longleftrightarrow 4d^2 - 8d - 4 \geq 0,
 $$
 which holds for $d \geq 3$. Now, for any $d \geq 2$, we proceed by induction on $n \geq 3$. In particular, if $d \geq 3$, {we use the case $n = 3$ for the base of induction, while, if $d = 2$, we use $n = 4$ as the base step}, which can be checked directly. Assume that formula \eqref{equation:Frob_implies_conjA_4} holds for $n$. Since 
 $$
 	2d+n-1 \leq \frac{(d+n-1)^2}{n-1} = \frac{d^2}{n-1} + 2d + n - 1,
 $$
 we conclude that \eqref{equation:Frob_implies_conjA_4} holds also for $n+1$. }
 
 This concludes the proof.
\end{proof}
\begin{proof}[Proof of Theorem \ref{thm:appendix_Froberg}] 
 Let $k,d,n\geq 2$ be integers. By Lemma \ref{Terracini} and \eqref{Terracini_codim}, we know that the dimension of the $s$-th secant variety of $V_{n,kd}^{(k)}$ is given by the dimension of the homogeneous part in degree $kd$ of the ideal $I = (g_1^{k-1},\ldots,g_s^{k-1})$, where the $g_i$'s are generic forms of degree $d$. Moreover, by Conjecture \ref{conj:C}, we have $\HS_{S/I}(t) = \left\lceil H(t) \right\rceil$, where $H(t) = \frac{(1-t^{d(k-1)})^s}{(1-t)^{n}} = \sum_{i\geq 0} H_i(t)$. {In particular,
 $$
   H_{kd} = 
  \begin{cases}
   {2d+n-1 \choose n-1} - s{d+n-1 \choose n-1} + {s \choose 2}, & \text{ for } k = 2; \\
   {kd+n-1 \choose n-1} - s{d+n-1 \choose n-1}, & \text{ for } k \geq 3.
  \end{cases}
 $$
  By Lemma \ref{lemma:1}, we have $\HF_{S/I}(kd) = \max\{0, H_{kd}\}$.} Thus, Conjecture \ref{conj:otta} directly follows from \eqref{Terracini_codim}.
\end{proof}
In particular, by the result of D.~Anick, we obtain the following result about $2$-ranks.
{
\begin{corollary}\label{thm:appendix_small_secants}
Conjecture \ref{conj:otta} holds in the case of sums of squares of ternary forms of even degree, i.e., 
$$
	\rk_2^\circ(3,2d) = \left\lceil \frac{{2d + 2 \choose 2}}{{d+2 \choose 2}} \right\rceil,
$$
except for $d = 1,3,4$, where it is $\left\lceil \frac{{2d + 2 \choose 2}}{{d+2 \choose 2}} \right\rceil + 1$.
\end{corollary}
\begin{remark}From the result of R. Stanley, we know that Fr\"oberg's Conjecture holds for $s \leq n+1$. Hence, from Terracini's Lemma and the proof of Theorem \ref{thm:appendix_Froberg}, we can conclude that, for $s \leq n+1$,
$$
	{\rm codim}~\sigma_s V_{n,kd}^{(k)} = \left\lceil\frac{(1-t^{d(k-1)})^s}{(1-t)^{n}}\right\rceil.
$$
\end{remark}
\begin{remark} In \cite{FOS}, the authors prove that, for any $n,k,d \geq 2$ positive integers, we have $\rk^\circ_k(n,kd) \leq k^n$ and, for $d \gg 0$, this bound is sharp. Hence, fixed $n$ and $k$, we have only finitely many cases left to compute the generic $k$-rank. In these case, if $n$ and $k$ are not too large, Fr\"oberg's Conjecture can be checked by computer. For example, it is possible to conclude that, 
$$
	\rk_2^\circ(4,2d) = \left\lceil \frac{{2d + 3 \choose 3}}{{d+3 \choose 3}} \right\rceil,
$$
except for $d = 1,2$, where $\rk_2^\circ(4,2d) = \left\lceil \frac{{2d + 3 \choose 3}}{{d+3 \choose 3}} \right\rceil + 1$. See \cite[Section 3.3.2]{One} for more details.
\end{remark}
}

{
\section{$k$-ranks of monomials}\label{appendix: monomials}
In this section, we continue the {study} started in \cite{CO} about $k$-ranks of monomials. 
\begin{proposition}
Let $m = x_1^{a_1} \cdots x_n^{a_n}$ be a monomial of degree $dk$ in $S$ and suppose that $ (k-2)n \leq d$. Then,
$\rk_k(m)\le k$.
\end{proposition}
\begin{proof}
It is enough to show that the monomial can be written as $m = m_1 m_2^{k-1}$ with $|m_1| = |m_2| = d$, since it follows from a classical well-known result, e.g., see \cite[Theorem 5.4]{Re2}, that the Waring rank of the monomial $x y^{k-1}$ is equal to $k$. \\
\indent For each $i$, write $a_i = q_i (k-1) + r_i$ with $0 \leq r_i \leq k-2$.
We have $a_1 + \cdots + a_n  = d(k-1) + d$, so $a_1 + \cdots + a_n$ is equivalent to $d$ modulo ($k-1$); that is, $r_1 + \cdots + r_n \equiv_{k-1} d$. From $r_1 + \cdots + r_n \leq (k-2)n$ and the assumption that $(k-2)n \leq d$, we get 
$r_1 + \cdots + r_n \leq d.$ It follows that 
$r_1 + \cdots + r_n + b (k-1) =  d$ for some non-negative integer $b$. \\
\indent We have $q_1+\cdots+q_n = (dk - (r_1+\cdots + r_n)/(k-1)$ and $b = (d-(r_1+\cdots r_n))/(k-1)$, so $q_1+\cdots+q_n \geq  b$. Thus, we can choose $b_1,\ldots,b_n$ such that
$b_1 + \cdots + b_n = b$ and such that $q_i \geq b_i$ for $i = 1, \ldots, n$.\\
\indent {Take} $m_1 = x_1^{r_1 + b_1 (k-1) } \cdots x_n^{r_n + b_n (k-1)}$  and let $m_2 = x_1^{ q_1 - b_1 } \cdots x_n^{q_n  - b_n}$.
{Then,} $m = m_1 m_2^{k-1}$ and we conclude the proof.
\end{proof}
\begin{example}
Let $n = 3, k = 4, d = 6$, i.e., ternary monomials of degree $24$ decomposed as sums of $4$-th {powers}. Consider $m = x_1^{3} x_2^{10} x_3^{11}.$ Hence, from the construction in the proof of {the} previous proposition, we have that 
$r_1 = 0, r_2 = 1, r_3 = 2, q_1 = 1, q_2 = 3, q_3 = 3$, $r_1 + r_2 + r_3 = 3$, so $b = 1$. Choose $b_2 = 1$.
Now, we write $m_1 = x_2^{4} x_3^2$, $m_2 = x_1 x_2^{2} x_3^{3}$ and $m = m_1 m_2^3$, so the rank is at most four.
\end{example}
}

\section{A canonical form for binary forms}\label{appendix: canonical}

The following results were obtained as a biproduct of the present study and they look similar to  Sylvester's original result and other known canonical forms, see \cite {Re1}. {Although they are not immediately relevant to our main topic, we include them here for the sake of completeness.}

\begin{theorem}\label{th:gen} Given positive integers $k\ge 2$ and $d\ge 1$, 

\medskip
\noindent
{\rm (i)} one can uniquely present a  general  binary form $p:=p(x,y)$ of degree  $kd$ as 
\begin{equation}\label{eq:gen}
p=p_0^k+y^dp_1^{k-1}+y^{2d}p_2^{k-2}+\dots +y^{(k-1)d}p_{k-1}^2+ y^{{(k-2)}d}p_{k-1}, 
\end{equation}
where every $p_j,\; j=0,2,\dots, k-2$ is a binary form of degree  $d$ with term $y^d$ missing and 
$p_{k-1}$ is a binary form of degree  $d$ (with no additional restrictions); each $p_i^{k-i}$ is uniquely defined which implies that $p_i$ itself is defined up to the $(k-i)$-th root of unity:

\medskip
\noindent
{\rm (ii)} one can  (non-uniquely) present any binary form $p:=p(x,y)$ of degree  $kd$ containing the monomial $x^{kd}$ as~\eqref{eq:gen},  
where every $p_j,\; j=0,2,\dots, k-1$ is a binary form of degree  $d$;

\end{theorem}

\begin{remark}
Observe that in case (i) the number of parameters in the right-hand side of \eqref{eq:gen} equals $kd+1$ which coincides with the dimension of the linear space of binary forms of degree  $kd$. In other words the right-hand side of \eqref{eq:gen} gives a canonical form of presentation of a general binary form of degree $kd$, comp. \cite{Re1}. 
\end{remark}

\begin{proof}

To prove (i), we will use induction on the power $k\ge 1$. 

Induction base. For $k=1$ and any $d\ge 1$, the statement is trivial since one can take $p= p_0$ in notation of Theorem~\ref{th:gen}.

Induction step. Assume that (i) is settled for all $d\ge 1$ and up to $k-1\ge 1$. Now given a binary form $p$ of degree $kd$, denote by $\widehat p=a_0x^{kd}+a_1x^{kd-1}y+\dots + a_{d-1}x^{(k-1)d+1}  y^{d-1}$ the form obtained by truncation of $p$ modulo all terms of the form $x^\ell y^{kd-\ell}$ with $\ell\ge d$. If $a_0\neq 0$ (which we can assume since $p$ is general), then one can find a form $p_0$ of degree $d$ such that 
$$p_0^k=\widehat p$$
modulo the same terms $x^\ell y^{kd-\ell}$ with $\ell\ge d$. Indeed, set 
$$p_0=a_0^{1/d}\left(x^d+\frac{1}{ka_0}(a_1 x^{d-1}y+a_2x^{d-2}y+\dots+a_{d-1}xy^{d-1})\right).$$
One can easily check that $p_0^k$ satisfies the above relation. Observe that $p-\widehat p$ is divisible  by $y^d$. For the quotient $(p-\widehat p)/q^d$ we obtain the same situation with $k$ substituted by $k-1$ and we can apply induction under the assumption that $(p-\widehat p)/y^d$ contains the term $x^{k-1}d$. Since we assume from the beginning that $p$ is general and our algorithm is deterministic, we can assume that the necessary terms are non-vanishing on each step of induction decreasing $k$ to $1$. (In fact the condition of generality can be written down rather explicitly in terms of non-vanishing of $k$ distinct discriminants.) 

(ii)  As we mentioned above, the only problem with the above representation under the assumption that $p$ contains $x^{kd}$ is that $(p-\widehat p)/y^d$ can miss the leading term $x^{k-1}d$. But if we allow to use $p_0$ with non-vanishing term $y^d,$ we can always obtain the latter condition. But then the choice of $p_0$ will be non-unique.  

\end{proof}

\begin{corollary}\label{cor:triv}
Any univariate polynomial $p(x)$ of degree at most $kd$ can be presented in the form 
$$\la x^{kd}+p_0^k(x)+p_1^{k-1}(x)+\dots +p_{k-2}^2(x)+ p_{k-1}(x),$$  
where $\la$ is a complex number and  $p_j(x),\; j=0,\dots , k-1$  are univariate polynomials of degree at most $d$.  \rm{(}Observe the similarity of the latter formula with \eqref{eq:even}.\rm{)}
\end{corollary}

\begin{proof}
If $p(x)$ is a polynomial of degree exactly $kd$, then  Corollary~\ref{cor:triv} is exactly statement (ii) of Theorem~\ref{th:gen}] in the non-homogeneous setting, i.e., for $y=1$. If $\deg p(x)< kd$, then adding a term $\la x^{kd}$, we obtain the previous situation. 
\end{proof}

\bibliographystyle{alpha}

\end{document}